\documentclass[11pt]{article}
\usepackage{enumerate}
\usepackage{tikz}
\usepackage{hyperref}
\hypersetup{
    pdftitle=   
    {Obstructions for matroids of path-width at most k and graphs of linear rank-width at most k},
   pdfauthor=  {Mamadou M. Kante, Eun Jung Kim, O-joung Kwon, and Sang-il Oum}
}
\usepackage{authblk}
\usepackage{amsthm,amssymb,amsmath}
\usepackage[margin=1.3in]{geometry}
\usepackage{hyphenat} %
\usepackage{lineno}
\newtheorem{theorem}{Theorem}[section]
\newtheorem{lemma}[theorem]{Lemma}
\newtheorem{proposition}[theorem]{Proposition}
\newtheorem{corollary}[theorem]{Corollary}

\newtheorem*{thm:main1}{Theorem~\ref{thm:main1}}
\newtheorem*{thm:main2}{Theorem~\ref{thm:main2}}

\newcommand\abs[1]{\lvert #1\rvert}
\newcommand\F{\mathbb F}
\newcommand\I{\mathcal I}

\newcommand\spn[1]{\langle #1\rangle}

\newcommand\tle\preccurlyeq
\newcommand\FS{\operatorname{FS}_k}
\newcommand\FST{\operatorname{FS}_{2k}}

\newcommand \tplus \oplus

\newcommand\pivot\wedge

\begin{document}
\title{Obstructions for matroids of path-width at most~$k$ and graphs of linear rank-width at most~$k$\thanks{An extended abstract appeared in \cite{KKKO2022}.}}

\author{Mamadou Moustapha Kant\'e\thanks{Supported by the grant from the French National Research Agency under JCJC program (ASSK: ANR-18-CE40-0025-01).}}
\affil{\small Université Clermont Auvergne, Clermont Auvergne INP, LIMOS, CNRS, Aubière, France.}
\author[$\dagger$2]{Eun Jung Kim%
}
\affil{\small Universit\'{e} Paris-Dauphine, PSL  University, CNRS, UMR 7243, LAMSADE, Paris, France.}
\author[3,4]{O-joung Kwon\thanks{Supported by the Institute for Basic Science (IBS-R029-C1).}\thanks{Supported by the National Research Foundation of Korea (NRF) funded by the Ministry of Education (No. NRF-2018R1D1A1B07050294)
and the Ministry of Science and ICT (No. NRF-2021K2A9A2A11101617).}}

\affil[3]{\small Department of Mathematics, Hanyang University, Seoul,~South~Korea.}

\affil[4]{\small Discrete Mathematics Group, Institute~for~Basic~Science~(IBS), Daejeon,~South~Korea.}

\author[$\ddagger$4,5]{Sang-il~Oum}
\affil[5]{\small Department of Mathematical Sciences, KAIST,  Daejeon,~South~Korea.}
\affil[ ]{E-mail addresses: \texttt{mamadou.kante@uca.fr},
\texttt{eunjungkim78@gmail.com}, \texttt{ojoungkwon@hanyang.ac.kr},
\texttt{sangil@ibs.re.kr}}

\date\today
\maketitle

\begin{abstract}
    Every minor-closed class of matroids of bounded branch-width can be characterized by a list of excluded minors, but unlike graphs, 
    this list may need to be infinite in general. 
    However, for each fixed finite field $\mathbb F$, 
    the list needs to contain only finitely many $\mathbb F$-representable matroids,
    due to the well-quasi-ordering of $\mathbb F$-representable matroids of bounded branch-width under taking matroid minors [J. F. Geelen,
    A. M. H. Gerards, and G. Whittle (2002)]. 
    But this proof is non-constructive and does not provide any algorithm for computing
    these $\mathbb F$-representable excluded minors in general. 

    We consider the class of matroids of path-width at most $k$ for fixed $k$.
    We 
    prove that for a finite field $\mathbb F$, every $\mathbb F$-representable excluded minor for the class of matroids of path-width at most~$k$ has at most
    $2^{|\mathbb{F}|^{O(k^2)}}$ elements.  
    We can therefore compute, for any integer $k$ and a fixed finite field $\mathbb F$, the set of $\mathbb F$-representable excluded minors for  the class of matroids of
    path-width $k$, and this gives as a corollary a polynomial-time algorithm for checking whether the path-width of an $\mathbb F$-represented matroid is at most $k$. 
    We also prove that every excluded pivot-minor for the class of graphs having linear rank-width at most $k$ has at most $2^{2^{O(k^2)}}$ vertices,
    which also results in a similar algorithmic consequence for linear rank-width of graphs.
\end{abstract}

\section{Introduction}
For a minor-closed class $\mathcal C$ of graphs or matroids, a graph or a matroid is an \emph{excluded minor} for $\mathcal C$ if it does not belong to $\mathcal C$ but all of
its proper minors belong to $\mathcal C$.

Robertson and Seymour~\cite{RS2004a} proved that every minor-closed class of graphs has finitely many excluded minors. 
This deep theorem has many algorithmic consequences for minor-closed classes of graphs. 
One of the corollaries is that 
for each minor-closed class $\mathcal I$ of graphs, 
there exists a  monadic second-order formula $\varphi_{\mathcal I}$
that expresses the membership in $\mathcal I$,
as there is a formula to decide whether a graph has a minor isomorphic to a fixed graph. 
However, the proof of Robertson-Seymour theorem is non-constructive and provides no  
algorithm for constructing the list of excluded minors
and therefore we only know the existence of $\varphi_{\mathcal I}$ and do not know how to construct $\varphi_{\mathcal  I}$ in general.

The class of graphs of path-width at most~$k$ is minor-closed and therefore 
the list of excluded minors for the class of graphs of path-width at most~$k$ is finite for each $k$. 
Actually, this is also implied by an earlier theorem of Robertson and Seymour~\cite{RS1990}, stating that graphs of bounded tree-width are well-quasi-ordered under taking minors. But this is still non-constructive.
In 1998, Lagergren~\cite{Lagergren1998} proved that each excluded minor for the class of graphs of path-width at most~$k$ has at most $2^{O(k^4)}$ edges.
Therefore we can now construct a monadic second-order formula $\varphi_k$ to decide whether the path-width of a graph is at most~$k$ for each $k$.
Since Courcelle's theorem~\cite{Courcelle1990} allows us to decide $\varphi_k$ on graphs of bounded tree-width in polynomial time, 
we obtain a polynomial-time algorithm to decide whether an input graph has path-width at most~$k$ for each fixed~$k$, 
even though a direct algorithm was proposed by Bodlaender and Kloks~\cite{BK1996}.

We aim to prove analogous theorems for the class of matroids of path-width at most~$k$ and for the class of graphs of linear rank-width at most~$k$.
For a matroid $M$ on the ground set $E(M)$, we define its connectivity function
$\lambda_M$ by \[\lambda_M(X)=r_M(X)+r_M(E(M)-X)-r(M)\quad\text{ for }X\subseteq E(M),\] where $r_M$ is the rank function of $M$.
The path-width of a matroid $M$ is defined as the minimum \emph{width} of  linear orderings of its elements, called \emph{path\hyp{}decompositions} or  \emph{linear layouts}, where the width of a path\hyp{}decomposition $e_1,e_2,\ldots,e_n$ is defined as the maximum of the values $\lambda_M(\{e_1,e_2,\ldots,e_i\})$ for all $i=1,2,\ldots,n$.

For matroid path-width, we do not yet know whether there are only finitely many excluded minors for the class of matroids of path-width at most~$k$. 
Previously, 
Koutsonas, Thilikos, and Yamazaki~\cite{KTY2011} showed a lower bound, proving that the number of excluded minors for the class of  matroids of path-width at most~$k$ is at least $(k!)^2$.
We remark that a class of matroids of bounded path-width 
is not necessarily well-quasi-ordered under taking minors;
Geelen, Gerards, and Whittle~\cite{GGW2002} showed that there is an infinite antichain of matroids of bounded path-width.

Geelen, Gerards, and Whittle~\cite{GGW2002} proved that for each finite field $\F$, $\F$\hyp{}representable matroids of bounded branch-width are well-quasi-ordered under taking minors, as a generalization of the theorem of Robertson and Seymour~\cite{RS1990} on graphs of bounded tree-width.
This implies that for each finite field $\F$, there are only finitely many $\F$\hyp{}representable excluded minors for the class of matroids of path-width at most~$k$. 

As a corollary, for each finite field $\F$ and an integer~$k$, there exists a monadic second-order formula $\varphi^\F_{k}$ to decide whether an $\F$-representable matroid has path-width at most~$k$, because one can write a monadic second-order formula to describe whether a matroid has a fixed matroid as a minor by Hlin\v{e}n\'y~\cite{Hlineny2006}.
Hlin\v{e}n\'y~\cite{Hlineny2006} also proved an analog of Courcelle's theorem for $\F$-represented matroids, 
showing  a fixed-parameter algorithm  to 
decide a monadic second-order formula 
on $\F$-represented matroids of bounded branch-width, for a finite field $\F$.
This allows us to conclude that there exists a fixed-parameter tractable algorithm to decide whether an input $\F$-represented matroid has path-width at most~$k$ by testing $\varphi^\F_k$.

However, the theorem of Geelen, Gerards, and Whittle~\cite{GGW2002} does not provide any method of constructing the list of  $\F$\hyp{}representable excluded minors
and so we did not know how to find $\varphi^\F_k$.
We are now ready to state our main theorem, showing an explicit upper bound of the size of every $\F$\hyp{}representable excluded minor.
\begin{theorem}\label{thm:main}
    For a finite field $\F$ and an integer $k$,
    each $\F$\hyp{}representable excluded minor for the class of matroids of path-width at most~$k$ has at most $2^{\abs{\F}^{O(k^2)}}$ elements.
\end{theorem}

Thus, by Theorem~\ref{thm:main}, we have an algorithm to construct $\varphi^\F_k$
and we have a fixed-parameter  algorithm to decide   whether an input $\F$\hyp{}represented matroid has path-width at most~$k$. 
Note that there is a subtle difference between ``have'' and ``there exist''; by Geelen, Gerards, and Whittle~~\cite{GGW2002}, we knew that there exists $\varphi^\F_k$, but we did not know how to construct it, because their proof is non-constructive. By Theorem~\ref{thm:main} we can enumerate all matroids of small size to find the list of all $\F$\hyp{}representable excluded minors
and therefore we can finally construct $\varphi^\F_k$.

We remark that Geelen, Gerards, Robertson, and Whittle~\cite{GGRW2003a} showed an analogous theorem for branch-width of matroids; for each $k\ge 1$, every excluded minor for the class of matroids of branch-width at most~$k$ has at most $(6^{k+1}-1)/5$ elements.\footnote{In \cite{GGRW2003a}, the connectivity function of matroids is defined to have $+1$, which makes $(6^k-1)/5$.}

By extending our method slightly, we also prove a similar theorem for the linear rank-width of graphs as follows.
\begin{theorem}\label{thm:main2}
    Each excluded pivot-minor for the class of graphs of linear rank-width at most~$k$ has at most $2^{2^{O(k^2)}}$ vertices.
\end{theorem}
Since every vertex-minor obstruction is also a pivot-minor obstruction, we deduce the following. 
\begin{corollary}\label{cor:main2}
    Each excluded vertex-minor for the class of graphs of linear rank-width at most~$k$ has at most $2^{2^{O(k^2)}}$ vertices.
\end{corollary}

The situation is very similar to that of matroids representable over a fixed finite field. 
Oum~\cite{Oum2004} showed that graphs of bounded rank-width  are well-quasi-ordered under taking pivot-minors,
which implies that the list of excluded pivot-minors for the class of graphs of linear rank-width at most~$k$ is finite. 
Again its proof is non-constructive and therefore it provides no algorithm to construct the list.
Jeong, Kwon, and Oum~\cite{JKO2013,JKO2014} proved that any list of excluded pivot-minors characterizing the class of graphs of linear rank-width at most~$k$ has at least $2^{\Omega(3^k)}$ graphs.

Corollary~\ref{cor:main2} answers an open problem of Jeong, Kwon, and Oum~\cite{JKO2014} on the number of vertices of each excluded vertex-minor for the class of graphs of linear rank-width at most~$k$. Adler, Farley, and Proskurowski~\cite{AFP2013} characterized excluded vertex-minors for the class of graphs of linear rank-width at most~$1$.
Theorem~6.1 of Kant\'e and Kwon~\cite{KK2015} implies that distance-hereditary excluded vertex-minors for the class of graphs of linear rank-width at most~$k$ have at most $O(3^k)$ vertices.

Previously, we only knew the existence of a modulo-$2$ counting monadic second-order formula $\Phi_k$ testing whether a graph has linear rank-width at most~$k$.
This is due to a theorem of Courcelle and Oum~\cite{CO2004} stating that for each graph $H$, there is a modulo-$2$ counting monadic second-order formula to decide
whether a graph has a pivot-minor isomorphic to $H$.  As there is a polynomial-time algorithm to decide a modulo-$2$ counting monadic second-order formula for
graphs of bounded rank-width (see~\cite[Proposition 5.7]{CO2004}), we can conclude that there exists a polynomial-time algorithm to decide whether an input
graph has linear rank-width at most~$k$. However, this algorithm is based on the existence of $\Phi_k$, and we did not know how to construct $\Phi_k$.  Finally, by
Theorem~\ref{thm:main2}, we know how to construct~$\Phi_k$ algorithmically.

Let us now explain the main ideas. We first observe that each excluded minor $M$ has path-width $k+1$, admits a \emph{linked path-decomposition}, which is
a path-decomposition satisfying some Menger-like condition, and each proper minor of $M$ has path-width at most $k$. Secondly, we
show that each excluded minor of sufficiently large size has many nested cuts, all of the same value. We finally show that among those cuts of the same
value, there are two nested cuts $X$ and $Y$ such that $M$ has a minor on $X\cup (E(M)\setminus Y)$ of path-width $k+1$, contradicting that all proper minors of $M$
have path-width at most $k$. 

One of the key ingredients in finding the minor is to use the data structure proposed by Jeong, Kim, and Oum~\cite{JKO2016}.  Based
on dynamic programming, they devised fixed-parameter algorithms to decide whether an $\F$\hyp{}represented matroid has path-width at most~$k$ and to decide
whether a graph has linear rank-width at most~$k$ without using the fact that there are only finitely many excluded minors.  Their so-called \emph{$B$-trajectories}
encode partial solutions which may be extended to the full solutions. Here is the idea behind $B$-trajectories. If $\lambda_M(X)=k$, then the
dimension of the vector space spanned by both $X$ and $E(M)\setminus X$ is exactly $k$. Since the underlying field is finite, this intersection subspace has
only finitely many subspaces. Combining this observation with the idea of \emph{typical sequences} appearing in Bodlaender and Kloks~\cite{BK1996}, Jeong, Kim,
and Oum~\cite{JKO2016} deduce that there are only finitely many collections, called the \emph{full sets}, of meaningful partial solutions (\emph{compact
  $B$-trajectories}) at every moment of the dynamic programming algorithm. We indeed prove that among all nested cuts ensured by the large size of $M$, there
are two nested cuts $X$ and $Y$ such that the full set associated with $Y$ can be obtained by applying the same linear transformation to all compact
$B$-trajectories of the full set associated with $X$, where $B$ is the vector space spanned by both $X$ and $E(M)\setminus X$. 

The second key ingredient of our
proof is the linking theorem for minors of matroids of Tutte~\cite{Tutte1965} and a corresponding theorem for pivot-minors of graphs by Oum~\cite{Oum2004}; both
are analogs of Menger's theorem.  These linking theorems will ensure that when two nested cuts display the identical full set up to a certain linear
transformation, one can obtain a proper minor or a proper pivot-minor having the same path-width or linear rank-width, respectively.

This paper is organized as follows.
Section~\ref{sec:prelim} reviews necessary definitions and known facts on matroids, branch\hyp{}decompositions, path\hyp{}decompositions, and Tutte's linking
theorem. We review in Section~\ref{sec:full} the data structure introduced in Jeong, Kim, and Oum~\cite{JKO2016}.
Section~\ref{sec:linked} presents a lemma on finding many cuts of the same width inside a \emph{linked} path\hyp{}decomposition.
We present the proof of the main theorem in Section~\ref{sec:proof}.
In Section~\ref{sec:linearrankwidth}, we present the proof for Theorem~\ref{thm:main2} on linear rank-width of graphs.
     
\section{Preliminaries}\label{sec:prelim}
For two sets $A$ and $B$, we write $A\triangle B$ to denote $(A-B)\cup (B-A)$.
\subsection{Matroids and minors}
A \emph{matroid} is a pair $(E,\I)$ of a finite set $E$ and a set $\I$ of subsets of $E$ satisfying the following three properties:
\begin{enumerate}[({I}1)]
    \item $\emptyset\in \I$.
    \item If $X\in \I$ and $Y\subseteq X$, then $Y\in\I$.
    \item If $X, Y\in \I$ and $\abs{X}<\abs{Y}$, then there is $e\in Y-X$ such that $X\cup \{e\}\in \I$.
\end{enumerate}
A subset of $E$ is \emph{independent} if it belongs to $\I$.
The \emph{ground set} of a matroid $M=(E,\I)$ is the set $E$ denoted by $E(M)$.
A subset of $E$ is \emph{dependent} if it is not independent.

Let $M=(E,\I)$ be a matroid on $n$ elements. 
We write $\I(M)$ to denote the set of independent sets of a matroid $M$.
A \emph{base} of a matroid is a maximal independent set.
A subset of $E$ is \emph{coindependent} if it is disjoint with some base.
The \emph{rank} of a set $X$ in a matroid $M$, denoted by $r_M(X)$, is the size of a maximal 
independent subset of $X$ in $M$.
The \emph{rank} of a matroid $M$ is $r(M):=r_M(E(M))$.
The \emph{connectivity function} of a matroid $M$, denoted by $\lambda_M$ is defined as 
\[ \lambda_M(X):=r_M(X)+r_M(E(M)-X)-r(M)\] 
for all $X\subseteq E(M)$.
It is easy to verify that $\lambda_M$ is \emph{submodular}, that is 
\[ 
    \lambda_M(X)+\lambda_M(Y)\ge \lambda_M(X\cup Y)+\lambda_M(X\cap Y)
\]
for all $X, Y\subseteq E(M)$. Also observe that $\lambda_M$ is \emph{symmetric}, that is $\lambda_M(X)=\lambda_M(E(M)-X)$ for all $X\subseteq E(M)$.

For $X\subseteq E$, the \emph{restriction} $M|_X$ of a matroid $M$ on $X$ is a matroid on the ground set $X$ such that $I\subseteq X$ 
is an independent set of $M|_X$ if and only if it is an independent set of $M$. The \emph{deletion} of $X$ from $M$ 
is the restriction of $M$ on $E-X$, denoted as $M\setminus X$. Another matroid operation is a \emph{contraction}. 
The contraction of $M$ by $X$, denoted as $M/X$, is a matroid with the ground set $E-X$
such that 
a set $I\subseteq E-X$ is an independent set of  $M/X$ if and only if there exists a base $B_X$ of $M|_X$ 
such that $I\cup B_X$ is an independent set of $M$.  Note that for $Y\subseteq E-X$, 
\[r_{M/X}(Y)=r_M(Y\cup X)-r_M(X),\] 
where $r_M$ is the rank function of a matroid $M$. 
For two matroids $M,N$, we say that $N$ is a \emph{minor} of $M$ if there exist disjoint subsets $C$ and $D$ of $E(M)$ such that 
$N=M\setminus D/C$.  
A minor $N$ of $M$ is \emph{proper} if $E(N)\neq E(M)$.

The following lemma is obtained easily from the above equation on the rank of a minor. 
\begin{lemma}[Geelen, Gerards, and Whittle~{\cite[(5.3)]{GGW2002}}]\label{lem:connminor}
    Let $M=(E,\I)$ be a matroid and 
    let $X$, $C$, $D$ be disjoint subsets of $E$. 
    Then $\lambda_{M\setminus D/C}(X)\le \lambda_M(X)$.
    Furthermore, equality holds if and only if 
    $r_M(X\cup C)=r_M(X)+r_M(C)$ and $r_M(E-X)+r_M(E-D)=r_M(E)+r_M(E-(X\cup D))$.
\end{lemma}

\subsection{Vector matroids}
One of the key examples of matroids is the class of vector matroids. 
Let $A$ be an $m\times n$ matrix over a field $\F$ whose columns are indexed by a set $E$ of column labels. 
Then a matroid $M(A)$ on $E$ can be defined from $A$ so that $X$ is independent in $M(A)$ if and only if the corresponding column vectors of $A$ are linearly independent. Such a matroid $M(A)$ is called a \emph{vector matroid}
and $A$ is called a \emph{representation} of the matroid $M(A)$. 
We say that a matroid $M$ is \emph{representable} over $\F$, or equivalently \emph{$\F$\hyp{}representable} if there is a matrix $A$ over $\F$ such that $M=M(A)$.
We say a matroid $M$ is \emph{$\F$\hyp{}represented} if it is given with its representation over $\F$.

Instead of using matrices, we may regard a vector matroid 
defined from  
a finite set of labeled vectors in a vector space, called a \emph{configuration} as in \cite{GGW2002}.
For a configuration $A$, we write $M(A)$ to denote the matroid on~$A$
such that a subset of $A$ is independent in $M(A)$ if and only if it is linearly independent in the underlying vector space.
Note that vectors in a configuration may coincide as we allow two different labels to represent the same vector.
We write $\spn{A}$ to denote the linear span of the vectors in $A$. 

\subsection{Path-width}
Let $E$ be a finite set with $n$ elements. 
A function $f:2^E\to\mathbb Z$ is \emph{submodular} if $f(X)+f(Y)\ge f(X\cup Y)+f(X\cap Y)$ for all $X,Y\subseteq E$
and is \emph{symmetric} if $f(X)=f(E-X)$ for all $X\subseteq E$.
We say that a function $f:2^E\to\mathbb Z$ is a \emph{connectivity function} 
if it is submodular, symmetric, and $f(\emptyset)=0$.

A \emph{linear layout} of $E$ is a permutation $\sigma=e_1,e_2,\ldots,e_n$ of $E$.
The \emph{width} of a linear layout $\sigma=e_1,e_2,\ldots,e_n$ with
respect to~$f$ is \[ \max_{1\leq i < n} f(\{e_1,e_2,\ldots,e_i\}).\] 
The \emph{path-width} of $f$ is the minimum width of all possible linear layouts  of $E$ with respect to~$f$.

For a matroid $M$, 
the linear layout of $E(M)$ is called a \emph{path\hyp{}decomposition} of $M$
and the path-width of $M$ is defined as the path-width of $\lambda_M$.

A linear layout $\sigma=e_1,e_2,\ldots,e_n$ 
is \emph{linked}
if 
for all $0\le i<j\le n$,
\[ 
    \min_{ \{e_1,e_2,\ldots,e_i\}\subseteq X
    \subseteq\{e_1,e_2,\ldots,e_j\}}
    f(X)
=\min_{i\le \ell\le j} f(\{e_1,e_2,\ldots,e_\ell\}).
\] 
Nagamochi~\cite{Nagamochi2012} presented an algorithm that runs in polynomial time for fixed $k$
to find a linear layout of width at most~$k$
if it exists for general connectivity functions. 
Although he did not state it explicitly, 
a key step of his algorithm, \cite[Lemma 2]{Nagamochi2012}, ensures that his algorithm outputs a linked linear layout, thus proving the following theorem.

\begin{theorem}[Nagamochi~\cite{Nagamochi2012}]\label{thm:nagamochi}
    If a connectivity function $f$ on $E$ has path-width $k$, 
    then $E$ has a linked linear layout of width at most~$k$.
\end{theorem}

\subsection{Tutte's linking theorem}
\begin{theorem}[Tutte~\cite{Tutte1965}]\label{thm:linking}
    Let $M$ be a matroid and $A$, $B$ be disjoint subsets of $E(M)$.
    Then 
    \[\lambda_M(X)\ge k \text{ for all }A\subseteq X\subseteq E(M)-B\]
    if and only if 
    $M$ has a minor $N$ on $A\cup B$ such that $\lambda_N(A)\ge k$.
\end{theorem}

For a configuration $A$ and $X\subseteq A$, 
let \[ \partial_A(X):=\spn{X}\cap \spn{A-X}.\]
Observe that $\lambda_{M(A)}(X)=\dim \partial_A(X)$.
The following proposition is essentially due to Geelen, Gerards, and Whittle~{\cite[(5.7)]{GGW2002}} and we modified their statement with almost the same proof.
Note that if $N=M/C\setminus D$ is a minor of $M$, then we can choose $D$ as a coindependent set in $M$ without changing $N$, see~\cite[Lemma 3.3.2]{Oxley2011a}. Thus it is easy to satisfy the requirements of the following proposition from Tutte's linking theorem.
\begin{proposition}\label{prop:stronglinking}
    Let $A$ be a configuration over a field $\F$ and 
    let $S$, $T$ be disjoint subcollections of $A$.
    Let $C$, $D$ be disjoint subcollections of $A$ such that
    $C\cup D=A-(S\cup T)$, 
    $D$ is coindependent in $M(A)$, and 
    for the minor $N=M(A)/C\setminus D$ of $M(A)$ on $S\cup T$, 
    \[ 
        \lambda_{N}(S)= \min_{S\subseteq X\subseteq A-T} \lambda_{M(A)}(X)=k.
    \]  
    Then for all subcollections $Z$ of $A$, 
    if $S\subseteq Z\subseteq A-T$ and $\lambda_{M(A)}(Z)=k$, then the following hold.
    \begin{enumerate}[(i)]
        \item For all $x,y\in \spn{Z}$, $x-y\in \spn{C}$ if and only if $x-y\in \spn {C\cap Z}$.
        \item For all $x,y\in \spn{A-Z}$, $x-y\in \spn{C}$ if and only if $x-y\in \spn{C-Z}$.
        \item For all $x,y\in \partial_A(Z)$, $x-y\in \spn{C}$ if and only if $x=y$.
        \item If $Z'$ is also a subcollection of $A$ such that $S\subseteq Z'\subseteq A-T$ and $\lambda_{M(A)}(Z')=k$, then for each $x\in\partial_A(Z')$, there is a unique $y\in \partial_A(Z)$ such that 
        $x-y\in \spn{C}$. Moreover, $x-y\in \spn{C\cap (Z\triangle  Z')}$. 
    \end{enumerate}
\end{proposition}

\begin{proof}
    Let $M=M(A)$.
    Since $D$ is coindependent, 
    $r_M(A-D)=r_M(A)$.
    Let $C_1=C\cap Z$, $D_1=D\cap Z$, $C_2=C-Z$, and $D_2=D-Z$.
    By Lemma~\ref{lem:connminor}, 
    \begin{align*}
        r_M(A-Z)+r_M(A-D_2)&=r_M(A)+r_M(A-(Z\cup D_2)),\\
        r_M(Z\cup C_2)&= r_M(Z)+r_M(C_2).
    \end{align*}
    As $r_M(A-D_2)=r_M(A)$, from the first equation, we have $r_M(A-Z)=r_M(A-(Z\cup D_2))=r_M(T\cup C_2)$ and so 
    \begin{equation}\label{eq:a-z1}
        \spn{A-Z}=\spn{T\cup C_2}.
    \end{equation}
    From the second equation, we have 
    \begin{equation}\label{eq:a-z2}
        \spn{Z}\cap \spn{C_2}=\{0\}.  
    \end{equation} 

    By symmetry between $S$ and $T$ and between $Z$ and $V-Z$, we have 
    \begin{equation}\label{eq:a-z3}
        \spn{Z}=\spn{S\cup C_1} \text{ and }\spn{A-Z}\cap \spn{C_1}=\{0\}.  
    \end{equation}

    Suppose that $x,y\in \spn{Z}$ and $x-y\in \spn{C}$. 
    Let $c_1\in \spn{C_1}$ and $c_2\in \spn{C_2}$ such that $x-y=c_1+c_2$. 
    Then $x-y-c_1\in \spn{C_2}\cap \spn{Z}$. By~\eqref{eq:a-z2}, $x-y-c_1=0$ and so $x-y\in \spn{C_1}$. This proves (i).
    By symmetry, (ii) is also proved. 

    By (i) and (ii), if $x,y\in \partial_A(Z)$ and $x-y\in \spn{C}$, then 
    $x-y\in \spn{C\cap Z}\cap \spn{C-Z}$. By~\eqref{eq:a-z2}, $\spn{C\cap Z}\cap \spn{C-Z}=\{0\}$ and therefore $x=y$. This proves (iii).

    To prove (iv), suppose that $x\in \partial_A(Z')$. By \eqref{eq:a-z1} applied to $Z'$, there exist $t\in \spn{T}$ and $c_2\in \spn{C- Z'}$ such that $x=t+c_2$.  
    Similarly, by \eqref{eq:a-z3}, there exist $s\in \spn{S}$ and $c_1\in\spn{C\cap Z'}$ such that $x=s+c_1$. 
    We can write $c_1= c_{11}+c_{12}$ for $c_{11}\in \spn{C\cap (Z\cap Z')}$ and $c_{12}\in \spn{C\cap (Z'-Z)}$
    and write $c_2=c_{21}+c_{22}$ for $c_{21}\in\spn{C\cap (Z-Z')}$
    and $c_{22}\in \spn{C-(Z\cup Z')}$. 
    Let us define $y=s+c_{11}-c_{21}=t+c_{22}-c_{12}$. 
    Then $y\in \partial_A(Z)$ because $s+c_{11}-c_{21}\in \spn{Z}$ and $t+c_{22}-c_{12}\in \spn{A-Z}$.
    Now observe that $x-y=c_{12}+c_{21}\in \spn{C\cap (Z\triangle  Z')}$. This proves that the desired $y$ exists. 
    By (iii), such $y$ is unique.
\end{proof}

\section{Full sets}\label{sec:full}
We review the concepts of $B$-trajectories and full sets introduced by Jeong, Kim, and Oum~\cite{JKO2016}.

\subsection{$B$-trajectories}\label{subsec:traj}
Let $B$ be a vector space. A \emph{statistic} is a triple $a=(L,R,\lambda)$ of subspaces $L$, $R$ of $B$ and a non-negative integer $\lambda$. 
For convenience, we write $L(a)=L$, $R(a)=R$, and $\lambda(a)=\lambda$. 
A \emph{$B$-trajectory} is a sequence $\Gamma=a_0,a_1,\ldots,a_n$ of statistics for a non-negative integer $n$ such that 
\begin{itemize}
    \item $R(a_0)=L(a_n)$, 
    \item $L(a_0)\subseteq L(a_1)\subseteq \cdots \subseteq L(a_n)\subseteq B$,
    \item $R(a_n)\subseteq R(a_{n-1})\subseteq \cdots \subseteq R(a_0)\subseteq B$.
\end{itemize}
The width of $\Gamma$ is $\max_{0\le i\le n} \lambda(a_i)$.
We write $\Gamma(i)$ to denote $a_i$. The \emph{length} of $\Gamma$, denoted by $\abs{\Gamma}$, is $n+1$.

Let $A=\{e_1,e_2,\ldots,e_n\}$ be a configuration over a field $\F$.
From a path\hyp{}decomposition $\sigma=e_1,e_2,\ldots,e_n$ of a represented matroid $M=M(A)$,
we can obtain its \emph{canonical $B$-trajectory} as follows.
For $i=0,1,2,\ldots,n$, let 
\begin{align*}
    L_i&=\spn{e_1,e_2,\ldots ,e_{i}}\cap B,\\ 
    R_i&=\spn{e_{i+1},e_{i+2},\ldots,e_n}\cap B, \text{ and}\\
    \lambda_i&=\dim \spn{e_1,e_2,\ldots,e_{i}}\cap \spn{e_{i+1},e_{i+2},\ldots,e_n}
    -\dim L_i\cap R_i.    
\end{align*}
Note that $L_0=R_{n}=\{0\}$ and $\lambda_0=\lambda_{n}=0$.
Let $a_i=(L_i,R_i,\lambda_i)$ for $i=0,1,2,\ldots,n$.
Then it is easy to see that $\Gamma=a_0,a_1,a_2,\ldots,a_{n}$ is a $B$-trajectory, which we call the \emph{canonical $B$-trajectory} of $\sigma$.
If $\Gamma$ is a canonical $B$-trajectory of some path\hyp{}decomposition $\sigma$ of $M=M(A)$,
then we say $\Gamma$ is \emph{realizable} in $A$.

For a $B$-trajectory $\Gamma=a_0,a_1,a_2,\ldots,a_n$,
the \emph{compactification} of $\Gamma$, denoted by $\tau(\Gamma)$,
is a $B$-trajectory obtained from $\Gamma$
by applying the following operations repeatedly until no further operations can be applied.
\begin{itemize}
    \item Remove an entry $a_i$ if $a_{i-1}=a_i$.
    \item Remove a subsequence $a_{i+1},a_{i+2}$, $\ldots$, $a_{j-1}$ if $i+1<j$, $L(a_i)=L(a_j)$, $R(a_i)=R(a_j)$, and 
    either $\lambda(a_i)\le \lambda(a_k)\le \lambda (a_j)$ for all $k\in \{i+1,i+2,\ldots,j-1\}$
    or $\lambda(a_i)\ge \lambda(a_k)\ge \lambda(a_j)$ for all $k\in \{i+1,i+2,\ldots,j-1\}$.
\end{itemize}
We say that a $B$-trajectory is \emph{compact} if $\tau(\Gamma)=\Gamma$.
Let $U_k(B)$ be the set of all compact $B$-trajectories of width at most~$k$.

\begin{lemma}[{Jeong, Kim, and Oum~\cite[Lemma 11]{JKO2016}}]\label{lem:numcompact}
    Let $B$ be a vector space over a finite field $\mathbb F$ with dimension $\theta$. Then 
    \[\abs{U_k(B)}\le 2^{9\theta+2}\abs{\mathbb F}^{\theta(\theta-1)} 2^{2(2\theta+1)k}.
     \]
\end{lemma}
We can define 
binary relations which compare two $B$-trajectories as follows~\cite{JKO2016}. 
For two statistics $a$ and $b$, we write $a\leq b$ if
\[L(a)=L(b),~R(a)=R(b),\text{ and } \lambda(a)\leq \lambda(b).\]
For two $B$-trajectories $\Gamma_1$ and $\Gamma_2$, we write $\Gamma_1\le\Gamma_2$ if the lengths of $\Gamma_1$ and $\Gamma_2$ are the same, say $n$, 
and $\Gamma_1(i)\le \Gamma_2(i)$ for all $0\le i\le n-1$.
A $B$-trajectory $\Gamma^*$ is called an \emph{extension} of a $B$-trajectory $\Gamma$ if $\Gamma^*$ can be obtained by repeating some statistics of $\Gamma$.
We say that $\Gamma_1 \tle \Gamma_2$ if
there are extensions $\Gamma_1^*$ of $\Gamma_1$ and $\Gamma_2^*$ of $\Gamma_2$ such that $\Gamma_1^*\leq \Gamma_2^*$.

\subsection{A full set}\label{subsec:full}

We review the \emph{full set} notion introduced by Jeong, Kim, and Oum~\cite{JKO2016}
used for their algorithm to decide the path-width of represented matroids. 
Let $A$ be a configuration of vectors in a vector space $V$ over a field $\F$. 
Let $B$ be a subspace of $V$.

The \emph{full set} of $A$ of width $k$ with respect to $B$,
denoted by $\FS(A, B)$, is the set of all compact $B$-trajectories $\Gamma$ 
of width at most~$k$ 
such that there exists a $B$-trajectory $\Delta$ realizable in $A$ with  $\Delta \tle \Gamma$.
From the definition, it is clear that 
\[ \text{$\FS(A,\{0\})\neq\emptyset$ if and only if $M(A)$ has path-width at most~$k$.}\] 

By Lemma~\ref{lem:numcompact}, 
the number of $B$-trajectories in $\FS(A,B)$ is bounded by a function of $\abs{\F}$, $\dim B$, and $k$.

The following two lemmas are an immediate consequence of Jeong, Kim, and Oum~\cite[Propositions 35 and 36]{JKO2016}.

\begin{lemma}\label{lem:shrink}
    Let $A$, $A'$ be configurations in a vector space $V$. 
    Let $k$  be a non-negative integer.
    Let $B$ be a subspace of $V$.
    If $\FS(A,B)=\FS(A',B)$, then 
    $\FS(A,\{0\})=\FS(A',\{0\})$.
\end{lemma}
\begin{lemma}\label{lem:fullset}
    Let $A_1$, $A_1'$, $A_2$, $A_2'$ be configurations in a vector space $V$. 
    Let $k$  be a non-negative integer.
    Let $B$ be a subspace of $V$  
    such that 
    $(\spn{A_1}+B)\cap (\spn{A_2}+B)=B$
    and 
    $(\spn{A_1'}+B)\cap (\spn{A_2'}+B)=B$.
    If $\FS(A_1,B)=\FS(A_1',B)$ and $\FS(A_2,B)=\FS(A_2',B)$, then 
    $\FS(A_1\cup A_2,B)=\FS(A_1'\cup A_2',B)$.
\end{lemma}

For a configuration $A=\{e_1,e_2,\ldots,e_n\}$ and 
a linear transformation $\phi$,
we write $\phi(A)$ to denote a configuration 
$\{\phi(e_1),\phi(e_2),\ldots,\phi(e_n)\}$.

If $B_1$ and $B_2$ are subspaces of the same dimension
and $\phi$ is a bijective linear transformation from $B_1$ to $B_2$, 
then for each $B_1$-trajectory $\Gamma$ we can define a $B_2$-trajectory $\Delta :=\phi(\Gamma)$ 
in the following way:
\[
L(\Delta(i))=\phi (L(\Gamma(i))), \quad R(\Delta(i))=\phi (R(\Gamma(i))), \quad \lambda(\Delta(i))=\lambda(\Gamma(i)),
\]
for every $0\leq i\leq \abs{\Gamma}-1$. For a set of $B$-trajectories $\mathcal R$, we define the set $\phi(\mathcal R)= \{\phi(\Gamma): \Gamma \in \mathcal R\}$.

Observe that if $\phi$ is a linear transformation on $\spn{A}$ that is injective on $\spn{A_1}$ and $B_1$ is a subspace of $\spn{A_1}$, 
then \[ \phi(\FS(A_1,B_1))=\FS(\phi(A_1),\phi(B_1)).\] 
Here on the right-hand side, we use $\phi$ values for all vectors in $\spn{A_1}$ but on the left-hand side, we only use $\phi$ for vectors in $B_1$.

We can deduce the following lemma easily from Lemmas~\ref{lem:shrink} and \ref{lem:fullset}.

\begin{lemma}\label{lem:fullsetkey}
    Let $k$ be a non-negative integer and let $\mathbb F$ be a field. 
    Let $A$ be a configuration in a vector space $V$ over $\mathbb F$
    and let $A'$ be a configuration in a vector space $V'$ over $\mathbb F$.
    Let $(A_1,A_2)$ be a partition of $A$
    and $(A_1',A_2')$ be a partition of $A'$. 
    If there is a bijective linear transformation $\phi:\partial_{A}(A_1)\to \partial_{A'}(A_1')$ such that 
    \begin{align*}
        \phi(\FS(A_1,\partial_A(A_1)))&=\FS(A_1',\partial_{A'}(A_1')) \text{ and }\\
        \phi(\FS(A_2,\partial_A(A_1)))&=\FS(A_2',\partial_{A'}(A_1')),
    \end{align*}
    then 
    the path-width of $M(A)$ is at most~$k$ if
    and only if the path-width of $M(A')$ is at most~$k$.
\end{lemma}
\begin{proof}
    We may assume that $\phi$ is the identity function on $\partial_A(A_1)$
    by applying some injective linear transformation on $V'$. 
    By extending $\phi$ and replacing both $V$ and $V'$ with $V+V'$, we may assume that $V=V'$ and $\phi$ is 
    the identity.
    Let $B=\partial_A(A_1)=\partial_{A'}(A_1')$. 

    Suppose that the path-width of $M(A)$ is at most~$k$. 
    Then $\FS(A,\{0\})$ is non-empty.
    Since $\FS(A_1,B)= \FS(A_1',B)$ and $\FS(A_2,B)=\FS(A_2',B)$, 
    by Lemma~\ref{lem:fullset}, 
    $\FS(A_1\cup A_2,B)=\FS(A_1'\cup A_2',B)$.
    By Lemma~\ref{lem:shrink}, 
    $\FS(A_1\cup A_2,\{0\})=\FS(A_1'\cup A_2',\{0\})$
    and therefore $\FS(A,\{0\})=\FS(A',\{0\})\neq\emptyset$.
    This implies that the path-width of $M(A')$ is at most~$k$. The converse holds by symmetry.
\end{proof}

\section{Finding many repeated cuts}\label{sec:linked}

The following lemma can be used to find many cuts in the linked path\hyp{}decomposition that are of the same width and linked.
\begin{lemma}\label{lem:cutcount}
    Let $\ell\ge 4$ be an integer.
    Let $a_0,a_1,a_2,\ldots,a_n$ be a sequence of integers such that 
    $a_i\ge a_0=a_n$ for all $0\le i\le n$ and $\abs{a_i-a_{i+1}}\le 1$.
    If \[n    \ge 
    \left(\ell-1+\frac{2(\ell-2)}{\ell-3}\right) (\ell-2)^{\max_{0\le i\le n}(a_i-a_0)} - \frac{2(\ell-2)}{\ell-3}
    ,
    \]
    then   there exist $0\le i_1<i_2<i_3<\cdots < i_\ell\le n $ and $w$ such that 
    \[a_{i_1}=a_{i_2}=\cdots=a_{i_\ell}=w
    \text{ and }
    a_i\ge w  \text{ for all }i_1\le i\le  i_\ell.\] 
\end{lemma}
\begin{proof}
    We proceed by induction on $M=\max_{0\le i\le n}(a_i-a_0)$.
    It  is trivial if $M=0$.
    Let $m=\abs{\{ i\in\{0,1,\ldots,n\}: a_i=a_0\}}$. 
    If $m\ge \ell$, then we are done.
    Thus we may assume that $m\le \ell-1$.
    Then there exists a subsequence $a_p,a_{p+1},\ldots,a_{q}$ such that 
    $a_{i}>a_0$ for all $p\le i\le q$, 
    and \[q-p+1\ge \frac{n}{m-1}-1\ge \frac{n}{\ell-2}-1.\]
    Equivalently, 
    \[ q-p+\frac{2(\ell-2)}{\ell-3}
    \ge \frac{1}{\ell-2} \left( n+ \frac{2(\ell-2)}{\ell-3}\right)
    \]
    and therefore  
    \[ q-p\ge \left(\ell-1+\frac{2(\ell-2)}{\ell-3}\right) (\ell-2)^{M-1} - \frac{2(\ell-2)}{\ell-3}.
    \]
    
    We may assume that $q-p$ is chosen as a maximum. Then 
    by the assumption that $\abs{a_i-a_{i+1}}\le 1$, we deduce that $a_p=a_q=a_0+1$. 
    Now we apply the induction hypothesis to the subsequence 
    $a_p, a_{p+1},\ldots,a_{q}$ to conclude the proof.
\end{proof}

We will apply Lemma~\ref{lem:cutcount} to a sequence $a_0,a_1,a_2,\ldots,a_{n}$
obtained from a linked path\hyp{}decomposition $\sigma=e_1,e_2,\ldots, e_n,$ 
where $a_i=\lambda_M(\{e_1,e_2,\ldots,e_{i}\})$ for $i=0,1,2,\ldots,n$. It is easy to verify that any path\hyp{}decomposition $\sigma$ of 
a represented matroid meets the requirement that $\abs{a_i-a_{i+1}}\le 1$ of Lemma~\ref{lem:cutcount}. 
The next lemma is needed.

\begin{lemma}
    Let $M$ be a matroid.
    If $e\in X\subseteq E(M)$, 
    then $\abs{\lambda_M(X)-\lambda_M(X-\{e\})}\le 1$. 
\end{lemma}
\begin{proof}
    By the submodularity of the connectivity function, we have 
    $\lambda_M(X-\{e\})+ \lambda_M(\{e\})\ge \lambda_M(X)$.
    Since $\lambda_M(\{e\})\le 1$, we have 
    $\lambda_M(X)\le \lambda_M(X-\{e\})+1$. 
    Since $\lambda_M$ is symmetric, 
    we deduce that 
    $\lambda_M(X-\{e\})\le \lambda_M(X)+1$.
\end{proof}

\section{The proof}	\label{sec:proof}

The following proposition proves Theorem~\ref{thm:main}.
\begin{proposition}\label{prop:main}
    Let $\F$ be a finite field and $k$ be a non-negative integer. 
    Let $M$ be an $\F$\hyp{}representable matroid of path-width larger than $k$.
    Let 
    $\ell = 2^{2^{9k+11}\abs{\mathbb F}^{k(k+1)} 2^{2(2k+3)k}}+1$.  
    If 
    \[
        \abs{E(M)}\ge 
        \left(\ell-1+\frac{2(\ell-2)}{\ell-3}\right) (\ell-2)^{k+1} - \frac{2(\ell-2)}{\ell-3},
    \]
    then there is $e\in E(M)$ such that $M/e$ or $M\setminus e$ has path-width larger than $k$. 
\end{proposition}

\begin{proof}
    Let $A$ be a configuration in a vector space over $\F$ such that $M=M(A)$.
    We may assume that $M\setminus e$ and $M/e$ has path-width at most~$k$ for every $e\in E(M)$. This implies that $M$ has path-width exactly $k+1$
    and by Theorem~\ref{thm:nagamochi}, there is a linked path\hyp{}decomposition $\sigma=e_1,e_2,\ldots,e_n$ of $M$ of width $k+1$.
    We identify $e_i$ with a vector in~$A$.
    
    For $i=0,1,2,\ldots,n$, let 
    $a_i=\lambda_M(\{e_1,e_2,\ldots,e_i\})$. 
    Then $0\le a_i\le k+1$ for all $i$. 

    By Lemma~\ref{lem:cutcount}, there exist integers $0\le t_1<t_2<\cdots<t_\ell\le n$ and $0\le \theta\le k+1$ such that 
    $a_{t_1}=a_{t_2}=\cdots=a_{t_\ell}=\theta$ 
    and $a_i\ge \theta$ for all $t_1\le i\le t_\ell$.
    Let $A_i=\{e_1,e_2,\ldots,e_{t_i}\}$ and $B_i=\partial_A(A_i)$ for $1\le i\le \ell$.

    Since $\sigma$ is a linked path\hyp{}decomposition, 
    $\lambda_M(X)\ge \theta$ for all 
    $A_1\subseteq X\subseteq A_\ell$. 
    By Theorem~\ref{thm:linking}, 
    there are  disjoint subcollections $C$, $D$ of $A$
    such that 
    $C\cup D=A-(A_1\cup (A-A_\ell))$ and 
    $\lambda_{M/C\setminus D}(A_1)=\theta$.
    We may assume that $D$ is coindependent, see~\cite[Lemma 3.3.2]{Oxley2011a}.
    Let $\pi:\spn{A}\to\spn{A}/\spn{C}$ be the linear transformation mapping $x\in \spn{A}$ to an equivalence class $[x]$ containing $x$ 
    where two vectors $x$ and $x'$ are equivalent if and only if $x-x'\in \spn{C}$.
    Let $B=\pi(\partial_A(A_1))$.

    By (iii) and (iv) of Proposition~\ref{prop:stronglinking}, 
    $\dim B=\theta$ and 
    $\pi(\partial_A(A_i))=\pi(\partial_A(A_j))$ for all $1\le i<j\le \ell$.

    Observe that $\pi(\FS(A_i,\partial_A(A_i)))\subseteq U_k(B)$.
    Since $\ell$ is big enough, by Lemma~\ref{lem:numcompact} and the pigeon-hole principle, 
    there exist $1\le i<j\le \ell$ such that 
    $\pi(\FS(A_i,\partial_A(A_i)))=\pi(\FS(A_j,\partial_A(A_j)))$.

    Let $C'=C\cap (A_{j}-A_i)$ and $D'=D\cap (A_j-A_i)$.
    Let $\phi: \spn{A}\to \spn{A}/\spn{C'}$
    be the linear transformation mapping $x\in \spn{A}$ 
    to an equivalence class containing $x$ 
    where two elements $x$, $y$ are equivalent if and only if $x-y\in\spn{C'}$.

    Let $B'=\phi(\partial_A(A_i))$.
    Since $C'\subseteq C$, by (iii)  of Proposition~\ref{prop:stronglinking},
    we have $\dim B'=\theta$. 
    Furthermore, from (iv) of Proposition~\ref{prop:stronglinking}, we deduce that 
    for $x\in \partial_A(A_i)$ and $y\in \partial_A(A_j)$, 
    $\pi(x)=\pi(y)$  if and only if $\phi(x)=\phi(y)$.
    Therefore,  $B'=\phi(\partial_A(A_j))$
    and 
    $\phi(\FS(A_i,\partial_A(A_i)))=\phi(\FS(A_j,\partial_A(A_j)))$.

    We claim that $\phi$ is an injection on $\spn{A_i}$. Suppose that $x,y\in \spn{A_i}$ and $x-y\in \spn{C'}=\spn{C\cap (A_j-A_i)}\subseteq \spn{A-A_i}$.
    Then $x-y\in \spn{C}$ and by (i) of Proposition~\ref{prop:stronglinking}, 
    we deduce that $x-y\in \spn{C\cap A_i}\subseteq \spn{A_i}$. 
    This would imply that $x-y\in \partial_A(A_i)$ and therefore $x=y$ by (iii) of Proposition~\ref{prop:stronglinking}. 
    By symmetry, we can also deduce that $\phi$ is an injection on $\spn{A-A_j}$.
    
    Let $N=M(A)/C'\setminus D'$.
    Then $A'=\phi(A_i\cup (A-A_j))$ is a configuration in the vector space $\spn{A}/\spn{C'}$ such that $N=M(A')$.
    Since $B'\subseteq \spn{\phi(A_i)}$  and $B'\subseteq \spn{\phi(A-A_j)}$, 
    we have $B'\subseteq \partial_{A'}(\phi(A_i))$. 
    By Lemma~\ref{lem:connminor}, $\dim  \partial_{A'}(\phi(A_i)) \le \theta$ and therefore $B'=\partial_{A'}(\phi(A_i))$.

    Since $\phi$ is an injection on $A_i$,
    \[ \FS(\phi(A_i),\partial_{A'}(\phi(A_i)))=\phi(\FS(A_j),\partial_{A}(A_j)).\] 
    Since $\phi$ is an injection on $A-A_j$, trivially 
    \[ \FS(\phi(A-A_j),\partial_{A'}(\phi(A-A_j)))=\phi(\FS(A-A_j),\partial_{A}(A-A_j)).\] 

    Since $N$ is a proper minor of $M$, the path-width of $N$ is at most~$k$. 
    By Lemma~\ref{lem:fullsetkey}, 
    $M$ has path-width at most~$k$ if and only if $N$ has path-width at most~$k$
    and therefore we deduce that the path-width of $M$ is at most~$k$, contradicting the assumption. 
\end{proof}

\section{Obstructions to linear rank-width}\label{sec:linearrankwidth}

\subsection{Basic definitions}

All graphs in this section are simple, having no loops and no parallel edges.

For a graph $G$,
the \emph{cut-rank} function $\rho_G$ of $G$
is defined as a function that maps a set $X$ of vertices of $G$
to the rank of the $X\times (V(G)-X)$ matrix over the binary field 
whose $ab$-entry is $1$ if and only if $a\in X$ is adjacent to $b\in V(G)-X$.
It is known that $\rho_G$ is symmetric and submodular, see Oum and Seymour~\cite{OS2004}, and therefore it is a connectivity function. 
We remark that $\rho_G(\emptyset)=\rho_G(V(G))=0$.
The \emph{linear rank-width} of a graph $G$ is defined to be the path-width of $\rho_G$.

For a pair $(x,y)$ of distinct vertices of a graph $G$, 
\emph{flipping} $(x,y)$ is an operation that adds an edge $xy$ if $x$, $y$ are non-adjacent in $G$ 
and deletes the edge $xy$ otherwise.
For an edge $uv$ of a graph $G$, 
we write $G\pivot uv$ to denote the graph $G'$ on $V(G)$ obtained by the following procedures.
\begin{enumerate}
    \item For every pair $x\in N(u)\cap N(v)$ and $y\in N(u)-N(v)$, 
    flip $(x,y)$.
    \item For every pair $x\in N(u)\cap N(v)$ and $y\in N(v)-N(u)$, 
    flip $(x,y)$.
    \item For every pair $x\in N(u)-N(v)$ and $y\in N(v)-N(u)$, 
    flip $(x,y)$.
    \item Swap the label of $u$ and $v$. 
\end{enumerate}
This operation is called the \emph{pivot}.
We remark that the purpose of the last operation  is to make $G\pivot uv \pivot vw = G\pivot uw$, see Oum~\cite{Oum2004}.
Here is an important property of pivots with respect to the cut-rank function. 
\begin{proposition}[See Oum~\cite{Oum2004}]
    If $H=G\pivot uv$, then $\rho_H(X)=\rho_G(X)$ for all $X\subseteq V(G)$.
\end{proposition}

We say that a graph $H$ is a \emph{pivot-minor} of a graph $G$
if $H$ is an induced subgraph of a graph obtained from $G$ by applying some sequence of pivots.
We say that a pivot-minor $H$ of $G$ is \emph{proper} if $V(H)\neq V(G)$.
Since deleting a vertex never increases the cut-rank function, 
we deduce the following easily from the previous proposition.
\begin{corollary}
    If $H$ is a pivot-minor of $G$, then the linear rank-width of $H$ is at most the linear rank-width of $G$.
\end{corollary}

\subsection{Tutte's linking theorem for pivot-minors}

Oum~\cite{Oum2004} proved an analog of Tutte's linking theorem for  pivot-minors. 
\begin{theorem}\label{thm:oumlink}
    Let $G$ be a graph and let $S$, $T$ be disjoint vertex sets of $G$. 
    Then there exists a pivot-minor $H$ on $S\cup T$ such that
    \[ 
        \rho_H(S)=\min_{S\subseteq X\subseteq V(G)-T} \rho_G(X).
    \] 
\end{theorem}

\subsection{From graphs to subspace arrangements}\label{subsec:graphtosubspace}
Let us now show how to represent a graph with a \emph{subspace arrangement}. A \emph{subspace arrangement} $\mathcal V$ over a field $\mathbb F$ is a finite set
of subspaces of a finite-dimensional vector space over $\mathbb F$.  We usually write a subspace arrangement as a family $\mathcal V=\{V_i\}_{i\in E}$ of
subspaces indexed by a finite set $E$.

A \emph{linear layout} of a subspace arrangement $\mathcal V$
is a permutation $\sigma=V_1,V_2,\ldots,V_n$ of $\mathcal V$.
The \emph{width} of a linear layout $\sigma =V_1,V_2,\ldots,V_n$ is equal to 
\[ 
\max_{1\le i<n} \dim \bigl((V_1+V_2+\cdots+V_i)\cap (V_{i+1}+V_{i+2}+\cdots+V_n)\bigr).
\] 
Note that this function is a connectivity function on $\mathcal V$.
The \emph{path-width} of $\mathcal V$ is the minimum width of linear layouts of $\mathcal V$.
If $\abs{\mathcal V}\le 1$, then we define the width of its linear layout to be $0$ and its path-width to be $0$.

As observed in \cite[Section VII]{JKO2016}, 
for a matroid $M$ represented by a configuration $A$, 
if we take $\mathcal V=\{ \spn{v}: v\in A\}$, 
then the path-width of $\mathcal V$ is equal to the path-width of $M(A)$.

We are now going to review the construction of a subspace arrangement from graphs that appeared in 
\cite[Section VIII]{JKO2016}.
This construction allows us to relay the concept of linear rank-width to the path-width of its corresponding subspace arrangement.
For a graph $G$ on the vertex set $\{1,2,\ldots,n\}$, 
let us define a subspace arrangement over the binary field as follows.
Let $\{e_1,e_2,\ldots,e_n\}$ be the standard basis of $\mathbb F_2^n$ where $\mathbb F_2$ is the binary field.
Let 
\[ v_i=\sum_{j\in N_G(i)} e_j,\]
where $N_G(i)$ denotes the set of neighbors of $i$.
Let $V_i=\spn{e_i,v_i}$
and let $\mathcal V_G=\{ V_i\}_{i\in V(G)}$.

Here is the key observation.
\begin{lemma}[Jeong, Kim, and Oum~{\cite[Lemma 52]{JKO2016}}]\label{lem:cutrank}
    For $X\subseteq V(G)$, 
    \[ 
    \dim \left((\sum_{i\in X} V_i) \cap (\sum_{j\in V(G)-X} V_j)\right) = 2\rho_G(X).
    \] 
\end{lemma}

\begin{corollary}
    The path-width of $\mathcal V_G$ is equal to twice the linear rank-width of $G$.
\end{corollary}

For a subset $X$ of $V(G)$,
let 
\begin{align*}
    I_X&=\{ e_i: i\in X\},\\
    A_X&=\{ v_i: i\in X\}, \text{ and }\\
    \partial_X&=\spn{I_X\cup A_X}\cap \spn{I_{V(G)-X}\cup A_{V(G)-X}}.
\end{align*}
By Lemma~\ref{lem:cutrank},  $\dim \partial_X=2\rho_G(X)$.
One can see that $I_Z$ is a set of some column vectors in the $n\times n$ identity matrix
and $A_Z$ is a set of some column vectors in the adjacency matrix of $G$.
Let $M_G$ be the binary matroid represented by the matrix $( I_n \,\, A(G) )$, where 
$I_n$ is the $n\times n$ identity matrix 
and $A(G)$ is the adjacency matrix of $G$.

Now, by applying Proposition~\ref{prop:stronglinking}, we deduce the following.
\begin{proposition}\label{prop:stronglinking2}
    Let $G$ be a graph and 
    let $S$, $T$ be disjoint sets of vertices of $G$
    such that 
    \[ \rho_{G[S\cup T]}(S)=\min_{S\subseteq X\subseteq V(G)-T} \rho_G(X)=k.\] 
    Let $C=V(G)-(S\cup T)$.
    Then for all subsets $Z$ of $V(G)$, 
    if $S\subseteq Z\subseteq V(G)-T$ and $\rho_G(Z)=k$, then the following hold.
    \begin{enumerate}[(i)]
        \item For all $x,y\in \spn{I_Z\cup A_Z}$, $x-y\in \spn{I_C}$ if and only if $x-y\in \spn {I_{C\cap Z}}$.
        \item For all $x,y\in \spn{I_{V(G)-Z}\cup A_{V(G)-Z}}$, $x-y\in \spn{I_C}$ if and only if $x-y\in \spn{I_{C-Z}}$.
        \item For all $x,y\in \partial_Z$, $x-y\in \spn{I_C}$ if and only if $x=y$.
        \item If $Z'$ is also a subset of $V(G)$ such that $S\subseteq Z\subseteq V(G)-T$ and $\rho_G(Z')=k$, then for each $x\in\partial_{Z'}$, 
        there is a unique $y\in \partial_Z$
        such that 
        $x-y\in \spn{I_C}$. Moreover $x-y\in \spn{I_{C\cap (Z\triangle  Z')}}$. 
    \end{enumerate}
\end{proposition}
\begin{proof}
    By Lemma~\ref{lem:cutrank}, 
    $\lambda_{M_G}(I_X\cup A_X)=2\rho_G(X)$ for all $X\subseteq V(G)$.
    Note that the dual matroid $M_G^*$  is represented by $( A_G ~ I_n )$ and therefore 
    $A_C$ is coindependent in $M_G$.
    Thus we can apply Proposition~\ref{prop:stronglinking} for $N=M_G/I_C \setminus A_C$.
\end{proof}

\subsection{Full sets for subspace arrangements}
In Subsection~\ref{subsec:full}, we reviewed the concept of full sets for the context of represented matroids or configurations. 
In fact, Jeong, Kim, and Oum~\cite{JKO2016} introduced full sets in more general form for subspace arrangements. 

Here we are going to show the difference compared to Subsections~\ref{subsec:traj} and \ref{subsec:full}.
For a subspace arrangement $\mathcal V$
and its linear layout $\sigma=V_1,V_2,\ldots,V_n$, 
the \emph{canonical $B$-trajectory} is defined as follows. For $i=0,1,\ldots,n$, let 
\begin{align*}
    L_i&=(\sum_{j=1}^i V_j) \cap B,\\
    R_i&=(\sum_{j=i+1}^n V_j)\cap B,\\
    \lambda_i&=\dim (\sum_{j=1}^i V_j)\cap (\sum_{j=i+1}^n V_j)  - \dim L_i\cap R_i, 
    \\
    a_i&=(L_i,R_i,\lambda_i).
\end{align*}
Then $\Gamma=a_0,a_1,a_2,\ldots,a_n$ is the \emph{canonical $B$-trajectory} of $\sigma$. 
We say that $\Gamma$ is realizable in $\mathcal V$ if it is a canonical $B$-trajectory of some linear layout of $\mathcal V$.

For a subspace arrangement $\mathcal V$, 
$\FS(\mathcal V,B)$ is defined as the set of all compact $B$-trajectories $\Gamma$ of width at most~$k$ such that there exists a $B$-trajectory $\Delta$ realizable in $\mathcal V$ with $\Delta\preceq \Gamma$.

Lemmas~\ref{lem:shrink} and \ref{lem:fullset} are special cases of the following two lemmas easily deduced from the result of  Jeong, Kim, and Oum~\cite[Propositions 35 and 36]{JKO2016}.
\begin{lemma}\label{lem:shrink2}
    Let $\mathcal V$, $\mathcal V'$ be subspace arrangements  over a field $\mathbb F$.
    Let $k$  be a non-negative integer.
    Let $B$ be a subspace of $\spn{\mathcal V\cup \mathcal V'}$.
    If $\FS(\mathcal V,B)=\FS(\mathcal V',B)$, then 
    $\FS(\mathcal V,\{0\})=\FS(\mathcal V',\{0\})$.
\end{lemma}
\begin{lemma}\label{lem:fullset2}
    Let $\mathcal V_1$, $\mathcal V_1'$, $\mathcal V_2$, $\mathcal V_2'$ be subspace arrangements over a field $\mathbb F$.
    Let $k$  be a non-negative integer.
    Let $B$ be a subspace of $\spn{\mathcal V_1\cup \mathcal V_2\cup \mathcal V_1'\cup \mathcal V_2'}$ 
    such that 
    $(\spn{\mathcal V_1}+B)\cap (\spn{\mathcal V_2}+B)=B$
    and 
    $(\spn{\mathcal V_1'}+B)\cap (\spn{\mathcal V_2'}+B)=B$.
    If $\FS(\mathcal V_1,B)=\FS(\mathcal V_1',B)$ and $\FS(\mathcal V_2,B)=\FS(\mathcal V_2',B)$, then 
    $\FS(\mathcal V_1\cup\mathcal V_2,B)=\FS(\mathcal V_1'\cup \mathcal V_2',B)$.
\end{lemma}
We can deduce the following lemma easily from Lemmas~\ref{lem:shrink2} and \ref{lem:fullset2}
by the same method of deducing 
Lemma~\ref{lem:fullsetkey} from Lemmas~\ref{lem:shrink} and \ref{lem:fullset}.
\begin{lemma}\label{lem:fullsetkey2}
    Let $k$ be a non-negative integer and let $\mathbb F$ be a field. 
    Let $\mathcal V$ be a subspace arrangement over $\mathbb F$
    and let $\mathcal V'$ be a subspace arrangement over $\mathbb F$.
    Let $(\mathcal V_1,\mathcal V_2)$ be a partition of $\mathcal V$
    and $(\mathcal V_1',\mathcal V_2')$ be a partition of $\mathcal V'$. 
    If there is a bijective linear transformation $\phi:\partial_{\mathcal V}(\mathcal V_1)\to \partial_{\mathcal V'}(\mathcal V_1')$ such that 
    \begin{align*}
        \phi(\FS(\mathcal V_1,\partial_{\mathcal V}(\mathcal V_1)))&=\FS(\mathcal V_1',\partial_{\mathcal V'}(\mathcal V_1')) \text{ and }\\
        \phi(\FS(\mathcal V_2,\partial_{\mathcal V}(\mathcal V_1)))&=\FS(\mathcal V_2',\partial_{\mathcal V'}(\mathcal V_1')),
    \end{align*}
    then 
    the path-width of $\mathcal V$ is at most~$k$ if
    and only if the path-width of $\mathcal V'$ is at most~$k$.
\end{lemma}
\subsection{Proof for linear rank-width}

\begin{proposition}
    Let $G$ be a graph of linear rank-width larger than $k$. 

    Let 
    $\ell = 2^{2^{18(k+1)+2+(2k+2)(2k+1)+2(4k+3)2k}}+1$.  
    If $G$ has more than
    \[
        \left(\ell-1+\frac{2(\ell-2)}{\ell-3}\right) (\ell-2)^{k+1} - \frac{2(\ell-2)}{\ell-3},
    \] vertices,
    then $G$ has a proper pivot-minor $H$ whose linear rank-width is larger than $k$.
\end{proposition}
\begin{proof}
    We may assume that $G$ has linear rank-width exactly $k+1$, because deleting a vertex decreases the linear rank-width by at most $1$. 
    Let us assume that $V(G)=\{1,2,\ldots,n\}$.

    By Theorem~\ref{thm:nagamochi}, there is a linked linear layout $\sigma$ of $G$ of width $k+1$.
    We may assume that $\sigma=1,2,\ldots,n$ by permuting vertices of $G$.
    For $i=0,1,2,\ldots,n$, let 
    $a_i=\rho_G(\{1,2,\ldots,i\})$. 
    Then $0\le a_i\le k+1$ for all $i$. 

    By Lemma~\ref{lem:cutcount}, there exist integers $0\le t_1<t_2<\cdots<t_\ell\le n$ and $0\le \theta\le k+1$ such that 
    $a_{t_1}=a_{t_2}=\cdots=a_{t_\ell}=\theta$ 
    and $a_i\ge \theta$ for all $t_1\le i\le t_\ell$.
    Let $S=\{1,2,\ldots,t_1\}$ and $T= \{t_\ell+1,t_\ell+2,\ldots,n\}$.

    By Theorem~\ref{thm:oumlink}, $G$ has a pivot-minor $G'$ on $S\cup T$ such that
    $\rho_{G'}(S) = \theta$. 
    Since pivoting does not change the cut-rank function, 
    we may assume that $G'$ is an induced subgraph of $G$ by applying pivots if necessary.

    Let $\mathcal V_G=\{ \spn{e_i,v_i}\}_{i\in V(G)}$ be the subspace arrangement as we constructed in Subsection~\ref{subsec:graphtosubspace}.
    Then $\mathcal V_G$ has path-width $2k+2$.
    Let $C= V(G)-(S\cup T)$.

    For $i=1,2,\ldots,\ell$, 
    let 
    \begin{align*}
        X_i&:=\{1,2,\ldots,t_i\}, 
        \\
        Y_i&:=\{t_i+1,t_i+2,\ldots,n\},\\ 
        \partial_i&:= \partial_{X_i}=\spn{I_{X_i}\cup A_{X_i}}\cap \spn{I_{Y_i}\cup A_{Y_i}}, \\
        \mathcal V_i&:= \{ \spn{ e_m, v_m}\}_{1\le m\le t_i}, \text{ and }\\
        \mathcal V_i'&:= \{ \spn{ e_m, v_m}\}_{t_i<m\le n}.
    \end{align*}
    By Lemma~\ref{lem:cutrank}, 
    $\dim \partial_i=2\theta$ for all $i=1,2,\ldots,\ell$.

    Let $\pi:\mathbb F_2^n \to \mathbb F_2^n / \spn{I_C}$ be the linear transformation that maps $x$ to an equivalence class containing $x$
    where two vectors are equivalent if their difference is in $\spn{I_C}$.
    Observe that if we identify $\mathbb F_2^n/\spn{I_C}$ with $\mathbb F_2^{n-\abs{C}}$ by ignoring coordinates indexed by $C$, then $\pi$ is a linear transformation that removes the coordinates indexed by $C$.
    Let $B=\pi( \partial_1)$. 

    By (iii) and (iv) of Proposition~\ref{prop:stronglinking2}, 
    $\dim B=2\theta$ and $\pi(\partial_i)=\pi(\partial_j)$ for all $1\le i<j\le \ell$.

    Observe that $\pi(\FST(\mathcal V_i,\partial_i))\subseteq U_{2k}(B)$.
    Since $\ell$ is big enough, 
    by Lemma~\ref{lem:numcompact} and the pigeon-hole principle, 
    there are $1\le i<j\le n$  such that 
    $\pi(\FST(\mathcal V_i,\partial_i))=\pi(\FST(\mathcal V_j,\partial_j))$. 

    Let $C'=\{ i+1,i+2,\ldots,j\}$. 
    Let $\phi:\mathbb F_2^n\to \mathbb F_2^n/\spn{I_{C'}}$ be the linear transformation that maps $x$ to an equivalence class containing $x$
    where two vectors are equivalent if their difference is in $\spn{I_{C'}}$.
    
    Let $B'=\phi(\partial_i)$.
    By (iii) of Proposition~\ref{prop:stronglinking2}, 
    $\dim B'=\dim \partial_i=2\theta$ because $C'\subseteq C$.
    By (iv) of Proposition~\ref{prop:stronglinking2}, 
    for $x\in \partial_i$ and $y\in \partial_j$, 
    \[ \pi(x)=\pi(y)  \text{ if and only if }\phi(x)=\phi(y).\]
    Therefore $B'=\phi(\partial_j)$
    and $\phi(\FST(\mathcal V_i,\partial_i))=\phi(\FST(\mathcal V_j,\partial_j))$.

    We claim that $\phi$ is an injection on $\spn{I_{X_i}\cup A_{X_i}}$.
    Suppose that $x,y\in \spn{I_{X_i}\cup A_{X_i}}$ and $x-y\in \spn{I_{C'}}$.
    Since $\spn{I_{C'}}\subseteq \spn{I_C}$, by (i) of Proposition~\ref{prop:stronglinking2}, 
    we deduce that $x-y\in \spn{I_{C\cap X_i}}$. 
    Since $C\cap X_i\subseteq X_i$ and $C'\subseteq Y_i$, we deduce that $x-y\in \partial_i$.
    By (iii) of Proposition~\ref{prop:stronglinking2}, 
    we have $x=y$ and therefore $\phi$ is an injection on  $\spn{I_{X_i}\cup A_{X_i}}$.
    By symmetry, we also deduce that $\phi$ is an injection on $\spn{I_{Y_i}\cup A_{Y_i}}$.

    Let $H=G-C_i$. 
    Since $B'\subseteq \spn{\phi(I_{X_i}\cup A_{X_i})}$ and $B'\subseteq \spn{\phi(I_{Y_j}\cup A_{Y_j})}$, 
    we have $B'\subseteq \spn{ \phi(I_{X_i}\cup A_{X_i})}
    \cap \spn{\phi(I_{Y_j}\cup A_{Y_j})}$.
    Since $\dim \spn{ \phi(I_{X_i}\cup A_{X_i})}
    \cap \spn{\phi(I_{Y_j}\cup A_{Y_j})} = 2\rho_H(X_i)\le 2\rho_G(X_i)=2\theta$
    and $\dim B'\ge \dim \partial_i =2\theta$, 
    we deduce that $B'= \spn{ \phi(I_{X_i}\cup A_{X_i})}
    \cap \spn{\phi(I_{Y_j}\cup A_{Y_j})}$.

    Since $\phi$ is an injection on $\spn{I_{X_i}\cup A_{X_i}}$, 
    \[ 
        \FST(\phi(\mathcal V_i),B')=\phi(\FST(\mathcal V_i,\partial_i))=\phi(\FST(\mathcal V_j,\partial_j)).
    \] 
    Since $\phi$ is an injection on $\spn{I_{Y_j}\cup A_{Y_i}}$, 
    we have 
    \[
        \FST(\phi(\mathcal V_j'),B')=\phi(\FST(\mathcal V_j',\partial_j)).
    \]
    By Lemma~\ref{lem:fullsetkey2}, 
    \( \FST(\phi(\mathcal V_i\cup\mathcal  V_j'),B')= \phi(\FST(\mathcal V_j\cup\mathcal V_j',\partial_j))=\phi(\FST(\mathcal V,\partial_j))\).
    By Lemma~\ref{lem:shrink2}, 
    \( 
        \FST(\phi(\mathcal V_i\cup\mathcal  V_j'),\{0\})= \phi(\FST(\mathcal V,\{0\}))
    \). 

    Since $H$ is a proper induced subgraph of $G$, the linear rank-width of $H$ is at most~$k$.  Note that $\phi$ is a linear transformation that omits coordinates corresponding to elements of $C'$ if we identify $\mathbb F_2^n/\spn{I_{C'}}$ with $\mathbb F_2^{n-\abs{C'}}$
    and therefore
    $\FST(\phi(\mathcal V_i\cup \mathcal V_j'),\{0\})$ is precisely the full set arising from the computation of the linear rank-width of $H$.
    Since $H$ has linear rank-width at most~$k$, 
    $\FST(\phi(\mathcal V_i\cup \mathcal V_j'),\{0\})$ is non-empty. 
    This implies that $\phi(\FST(\mathcal V,\{0\}))$ is non-empty, and so $\mathcal V$ has path-width at most $2k$
    and $G$ has linear rank-width at most~$k$, contradicting the assumption. 
\end{proof}

\end{document}